\documentclass{amsart}

\title[Elliptic curves with with prescribed mod~$3$ representation]
{Constructing families of elliptic curves with prescribed mod~$3$ representation via Hessian and Cayleyan curves}
\thanks{After submitting the first version of this paper on the arXiv, the author was informed that the main observation of this paper had already been made by Tom Fisher, \emph{The Hessian of a genus one curve}, Proc. Lond. Math. Soc. (3) \textbf{104} (2012), 613-648.  
}
\author[M. Kuwata]{Masato Kuwata}
\address{Faculty of Economics, Chuo University, 742-1 Higashinakano, 
Hachioji-shi, Tokyo 192-0393, Japan}
\email{kuwata@tamacc.chuo-u.ac.jp}

\date{}%\today}

\usepackage{mathrsfs}%\usepackage{mathabx}
\usepackage[dvips]{graphicx}
\usepackage[all]{xy}

\newtheorem{theorem}{Theorem}[section]
\newtheorem{lemma}[theorem]{Lemma}
\newtheorem{prop}[theorem]{Proposition}
\newtheorem{cor}[theorem]{Corollary}

\theoremstyle{definition}
\newtheorem{definition}[theorem]{Definition}

\theoremstyle{remark}
\newtheorem{remark}[theorem]{Remark}

\renewcommand{\P}{\mathbf{P}}
\def\dP^#1{(\P^{#1})^{*}}
\newcommand{\E}{\mathcal{E}}
\newcommand{\F}{\mathcal{F}}

\newcommand{\Z}{\mathbf{Z}}
\newcommand{\He}{\operatorname{{\it He}}}
\newcommand{\Cay}{\operatorname{{\it Ca}}}

\newcommand{\vect}[1]{\boldsymbol{#1}}
\newcommand{\Gal}{\operatorname{Gal}}
\newcommand{\Aut}{\operatorname{Aut}}
\newcommand{\rank}{\operatorname{rank}}
\newcommand{\Dir}{\nabla}
\newcommand{\<}{\langle}\renewcommand{\>}{\rangle}
\newcommand{\tr}{{}^{t}}
\numberwithin{equation}{section}

\begin{document}
\maketitle
\section{Introduction}

Let $E_{0}$ be an elliptic curve defined over a number field~$k$. 
The subgroup of $3$-torsion points $E_{0}[3]$ of~$E_{0}(\bar k)$ is a Galois module that gives rise to a representation 
\[
\bar\rho_{E_{0},3}:\Gal(\bar k/k)\to \Aut(E_{0}[3])\cong GL_{2}(\mathbb{F}_{3}).
\]
The collection of elliptic curves over~$k$ having the same mod $3$ representation as a given elliptic curve $E_{0}$ forms an infinite family.  In this paper we give an explicit construction of this family using the notion Hessian and Cayleyan curves in classical geometry. 

Suppose $\phi:E_{0}[3]\to E[3]$ is an isomorphism as Galois modules.  Then, either $\phi$ commutes with the Weil pairings $e_{E_{0},3}$ and 
$e_{E,3}$, or we have 
\[
e_{E,3}\bigl(\phi(P),\phi(Q)\bigr)=e_{E_{0},3}(P,Q)^{-1}
\]
for all $P, Q\in E_{0}[3]$.  In the former case, we call $\phi$ a symplectic isomorphism, or an isometry.  In the latter case, we call $\phi$ an anti-symplectic isomorphism, or an anti-isometry.

Rubin and Silverberg \cite{Rubin-Silverberg} gave an explicit construction of the family of elliptic curves $E$ over $k$ that admits a symplectic isomorphism $E_{0}[3]\to E[3]$.  There is a universal elliptic curve $\E_{t}$ over a twist of noncompact modular curve $Y_{3}$ which is a twist of the Hesse cubic curve
\(
x^{3}+y^{3}+z^{3}=3\lambda xyz.
\)
We give an alternative construction of this family using the Hessian curve of $E_{0}$ (see \S2 for definition).   
The family of elliptic curves $F$ over $k$ that admits a anti-symplectic isomorphism $E_{0}[3]\to F[3]$ is related to the construction of curves of genus~$2$ that admit a morphism of degree~$3$ to~$E_{0}$ (see Frey and Kani \cite{Frey-Kani}).  We show that there is a universal elliptic curve $\F_{t}$ for this family, and we give a construction using the Caylean curves (see \S3) in the dual projective plane.   

Our main results are roughly as follows. 
Choose a model of $E_{0}$ as a plane cubic curve such that the origin $O$ of the group structure of $E_{0}$ is an inflection point.  A Weierstrass model of $E_{0}$ satisfies this condition.  Then, its Hessian curve $\He(E_{0})$ is a cubic curve in the same projective plane and the intersection $E_{0}\cap \He(E_{0})$ is nothing but the group of $3$-torsion points. We will show that the pencil of cubic curves
\[
\E_{t}: E_{0} + t \He(E_{0})
\]
is nothing but the family with symplectic isomorphism $E_{0}[3]\to \E_{t}[3]$, as the nine base points of the pencil form the subgroup $\E_{t}[3]$ for each~$t$ 
(Theorem~\ref{thm:4-1}).

It is classically known that the Hessian curve $\He(E)$ admits a fixed point free involution~$\iota$ (see \cite{Artebani-Dolgachev}\cite{Dolgachev}).  The line joining the points $P$ and $\iota(P)$, denoted by $\overline{P\iota(P)}$, gives a point in the dual projective plane $\dP^{2}$.  The locus of such lines $\overline{P\iota(P)}$ for all $P\in E_{0}$ is a cubic curve in $\dP^{2}$ classically known as Cayleyan curve and it is denoted by $\Cay(E_{0})$. It is easy to see that $\Cay(E_{0})$ is isomorphic to the quotient $\He(E_{0})/\<\iota\>$, and we can prove that the map associating $P\in E_{0}[3]$ to its inflection tangent $T_{P}\in \Cay(E_{0})[3]$ is an anti-symplectic isomorphism. Since $\Cay(E_{0})$ also has a fixed point free involution, we may expect that it is the Hessian of a cubic curve in $\dP^{2}$, and it turns out this is the case.  There is a cubic curve $F_{0}$ in $\dP^{2}$ whose Hessian is $\Cay(E_{0})$.  The pencil of cubic curves
\[
\F_{t}: F_{0} + t \Cay(E_{0})
\]
is then the family with anti-symplectic isomorphism $E_{0}[3]\to \F_{t}[3]$.

If $E_{0}$ is given by the Weierstrass equation 
\[
E_{0}:y^{2}z = x^{3} + Axz^{2} + B z^{3},
\]
then the equations of $\He(E_{0})$, $\Cay(E_{0})$, and $F_{0}$ are given by
\[\renewcommand{\arraystretch}{1.5}
\begin{array}{ccl}
\He(E_{0})&:&3Ax^2z+9Bxz^2+3xy^2-A^2z^3=0,
\\
\Cay(E_{0})&:&A\xi^3 + 9B\xi\eta^2 + 3\xi\zeta^2 - 6A\eta^2\zeta=0,
\\
F_{0} &:&AB\xi^3 - 2A^2\xi^2\zeta - (4A^3+27B^2)\xi\eta^2 - 9B\xi\zeta^2 + 2A\zeta^3=0,
\end{array}\]
where $(\xi:\eta:\zeta)$ is the dual coordinate of $\dP^{2}$.

In \S6 we give some applications.  From our description of $\E_{t}$ and $\F_{t}$, it is clear that the elliptic surfaces associated to these families are rational elliptic surfaces over~$k$.  As a consequence, we are able to apply Salgado's theorem \cite{Salgado} to our family (Theorem~\ref{thm:salgado}).  

Let $F$ be a nonsingular member of $\F_{t}$.  Since we have an anti-symplectic isomorphism $\psi: E_{0}[3]\to F[3]$, Frey and Kani \cite{Frey-Kani} show that there exists a curve $C$ of genus~$2$ that admits two morphism $C\to E_{0}$ and $C\to F$ of degree~$3$. Indeed, the quotient of $E_{0}\times F$ by the graph of $\psi$ is a principally polarized abelian surface that is the Jacobian of a curve $C$ of genus~$2$.
For example, if we take $\Cay(E_{0})$ as $F$, then it turns out that this is the degenerate case where $C\to \Cay(E_{0})$ is ramified at one place with ramification index~$3$.  We will give explicit formulas for this case (Proposition~\ref{prop:genus-2}).

%Acknowledgements.

\section{Hessian of a plane cubic}

Let $C$ be a plane curve defined by a homogenous equation $F(x,y,z)=0$.   In this section we summarize some facts on polarity, with a special emphasis on our particular case of cubic curves.  In this section and the next, the base filed is taken as an algebraic closure of~$k$.  For more general treatment,  see Dolgachev \cite[Ch.~1 and Ch.~3]{Dolgachev}.

For a nonzero vector $\vect{a}=\tr(a_{0},a_{1},a_{2})$, we define the differential operator$\Dir_{\vect{a}}$ by
\[
\Dir_{\vect{a}} =a_{0}\frac{\partial}{\partial x}
+a_{1}\frac{\partial}{\partial y}
+a_{2}\frac{\partial}{\partial z}.
\]
Here, $\Dir_{\vect{a}}F(x,y,z)$ stands for the directional derivative of $F(x,y,z)$ along the direction vector~$\vect{a}$. 
The \emph{first polar} curve of $C$ is then defined by the equation 
\(
\Dir_{\vect{a}}F(x,y,z)=0
\).
It depends only on the point $a=(a_{0}:a_{1}:a_{2})\in \P^{2}$, but not the vector $\vect{a}$ itself. Thus, we denote the first polar curve by $P_{a}(C)$:
\[
P_{a}(C) : \Dir_{\vect{a}}F(x,y,z)=0.
\]
When $C$ is a cubic, $P_{a}(C)$ is a conic.

\begin{figure}
\includegraphics[scale=0.36]{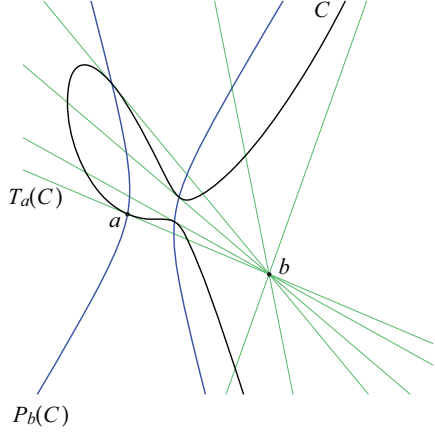}
\caption{First polar curve}\label{fig:1}
\end{figure}

Also, the \emph{second polar} curve of $C$ is defined by $\Dir_{\vect{a}}\Dir_{\vect{a}}F(x,y,z)=0$.  The composition of differential operators $\Dir_{\vect{a}}\circ \Dir_{\vect{a}}$ is sometimes denote by $\Dir_{\vect{a}^{2}}$, and thus the second polar is denoted by $P_{a^{2}}(C)$:
\[
P_{a^{2}}(C) : \Dir_{\vect{a}^{2}}F(x,y,z)= 0.
\]
When $C$ is a cubic, $P_{a^{2}}(C)$ is a line.

In general we define the differential operator $\Dir_{\vect{a}^{k}}$ inductively by  
\[
\Dir_{\vect{a}^{k}}=\Dir_{\vect{a}}\circ \Dir_{\vect{a}^{k-1}}, \quad k\ge 2,
\]
and for a plane curve $C$ of any degree, the $k$-th polar $P_{a^{k}}(C)$ is defined by 
\[
P_{a^{k}}(C) : \Dir_{\vect{a}^{k}}F(x,y,z)=0.  
\]
With this notation the Taylor expansion formula for a general analytic function $F$ can be written in the following form:
\begin{align*}
F(\vect{x}+\vect{a}) 
&=\sum_{k=0}^{\infty} \frac{1}{k!}\Dir_{\vect{a}^{k}}F(\vect{x})
\\
&= F(\vect{x}) + \Dir_{\vect{a}}F(\vect{x})
+\frac{1}{2!}\Dir_{\vect{a}^{2}}F(\vect{x}) 
+\frac{1}{3!}\Dir_{\vect{a}^{3}}F(\vect{x}) 
+\cdots,
\end{align*}
where $\vect{x}=\tr(x,y,z)$.  
With matrix notation we can write
\[
\Dir_{\vect{a}}F(\vect{x}) = F'(\vect{x})\vect{a},
\quad
\Dir_{\vect{b}}\Dir_{\vect{a}}F(\vect{x}) = \tr\vect{b}F''(\vect{x})\vect{a}.
\]
where
\[
F'(\vect{x}) =
\begin{pmatrix} F_{x}(\vect{x})  & F_{y}(\vect{x})  & F_{z}(\vect{x}) 
\end{pmatrix},
\quad
F''(\vect{x}) 
=\begin{pmatrix} 
F_{xx}(\vect{x})  & F_{xy}(\vect{x})  & F_{xz}(\vect{x})  \\
F_{yx}(\vect{x})  & F_{yy}(\vect{x})  & F_{yz}(\vect{x})  \\
F_{zx}(\vect{x})  & F_{zy}(\vect{x})  & F_{zz}(\vect{x}) 
\end{pmatrix}.
\]
Here $F_{x}(\vect{x})$ and $F_{xx}(\vect{x})$, for example, mean partial derivative $\frac{\partial F}{\partial x}(x,y,z)$ and the second partial derivative $\frac{\partial^{2}F}{\partial x^{2}}(x,y,z)$.
The $3 \times 3$ matrix $F''(\vect{x})$ is called the \emph{Hessian matrix}, and is sometimes denoted by $\He(F)(\vect{x})$. 

\begin{definition}
The \emph{Hessian} $\He(C)$ of $C$ is defined by
\[
\He(C):
\det F''(\vect{x})=\left|
\begin{matrix} 
F_{xx}(\vect{x}) & F_{xy}(\vect{x}) & F_{xz}(\vect{x}) \\
F_{yx}(\vect{x}) & F_{yy}(\vect{x}) & F_{yz}(\vect{x}) \\
F_{zx}(\vect{x}) & F_{zy}(\vect{x}) & F_{zz}(\vect{x})
\end{matrix}
\right|
=0.
\]
\end{definition}

\begin{remark}
The determinant $\det F''(\vect{x})$ can be identically zero.  For example, if $C$ is defined by $F(x,y,z)=xy^{2}+zy^{2}$, then $\det F''(\vect{x})=0$ identically.  In this case we have $\He(C)=\P^{2}$.  
%Otherwise $\He(C)$ is a cubic curve, though it can be singular.  The Hessian of the Fermat curuve $x^{3}+y^{3}+z^{3}=0$ is the union of three lines.  It can be proved that if $\He(C)$ is reducible, it decomposes a union of lines.
\end{remark}

To see the meaning of $P_{a^{k}}(C)$, let $a=(a_{0}:a_{1}:a_{2})$ and $b=(b_{0}:b_{1}:b_{2})$ be  two points in $\P^{2}$, and let $\ell=\overline{ab}$ be the line joining the two points.  $\ell$ is the image of the map $\lambda:\P^{1}\to\P^{2}$
\[
\lambda:(s:t)\mapsto sa + tb =
(sa_{0}+tb_{0}:sa_{1}+tb_{1}:sa_{2}+tb_{2}).
\]

If the degree of $F$ is $n$, then the Taylor expansion formula around the point $(s:t)=(1:0)$ gives a homogenous polynomial of degree~$n$ in $s$ and~$t$:
\begin{multline}\label{eq:Taylor-n}
F(s\vect{a}+t\vect{b})=
\\
F(\vect{a}) s^{n} 
+ \Dir_{\vect{b}}F(\vect{a}) s^{n-1}t
+\frac{1}{2!}\Dir_{\vect{b}^{2}}F(\vect{a}) s^{n-2}t^{2}
+\frac{1}{3!}\Dir_{\vect{b}^{3}}F(\vect{a}) s^{n-3}t^{3} + \cdots.
\end{multline}

Before going further, we need a few lemmas.

\begin{lemma}\label{lem:conic}
Let $Q$ be a plane conic defined by $\tr\vect{x}M\vect{x}=0$, where $M$ is a symmetric matrix.  Let $a=(a_{0}:a_{1}:a_{2})$ be a point on~$Q$, and $\vect{a}=\tr(a_{0},a_{1},a_{2})$ is the corresponding vector.
\begin{enumerate}
\item $a$ is a singular point if and only if $M\vect{a}=\vect{0}$.
\item $Q$ is degenerate if and only if $\det M=0$.
\item If $a$ is a smooth point, then the tangent line $T_{a}(Q)$ at $a$ is given by the equation $\tr\vect{x}M\vect{a}=0$.
\end{enumerate}
\end{lemma}

\begin{proof}
Let $b$ be a point in $\P^{2}$ different from $a$.
The intersection between the line $\overline{ab}$ and $Q$ is given by the solution $(s:t)$ to the equation
\begin{equation}\label{eq:int_ab}
\tr(s\vect{a}+t\vect{b})M(s\vect{a}+t\vect{b})
=2(\tr\vect{b}M\vect{a})st + (\tr\vect{b}M\vect{b})t^{2}=0.
\end{equation}
The point $a$ is a singular point of $Q$ if and only if the multiplicity of intersection $\overline{ab}\cap Q$ at $a$ is greater than~$1$ for any point $b$ in $\P^{2}$.  This is equivalent to say that $\tr\vect{b}M\vect{a}=0$ for any $\vect{b}$. This condition in turn is equivalent to the condition $M\vect{a}=\vect{0}$. 

The curve $Q$ is degenerate if and only if it has a singular point.  This is equivalent to the existence of a nonzero vector $\vect{a}$ satisfying $M\vect{a}=\vect{0}$, which in turn is equivalent to the condition $\det M=0$.

A point $b$ is on the tangent line $T_{a}$ if and only if $\overline{ab}$ intersects with $Q$ at $a$ with multiplicity~$2$. The equation \eqref{eq:int_ab} shows that 
the latter condition is equivalent to the condition that $b$ satisfies the equation $\tr\vect{x}M\vect{a}=0$.  Thus, the equation of the tangent line $T_{a}$ is given by $\tr\vect{x}M\vect{a}=0$.
\end{proof}

\begin{lemma}[Euler's formula]\label{lem:euler}
Let $F(\vect{x})$ be a homogeneous polynomial of degree~$d$, where $\vect{x}=\tr(x_{0},x_{1},\dots,x_{n})$.  Then, we have
\[
d(d-1)\dots(d-k+1)F(\vect{x})
=\Dir_{\vect{x}^{k}}F(\vect{x}).
\]
\end{lemma}

\begin{proof}
Since $F$ is homogeneous of degree~$d$, we have $F(\lambda\vect{x})=\lambda^{d}F(\vect{x})$.  Take its $k$th derivative with respect to $\lambda$ and put $\lambda=1$.
\end{proof}

\begin{prop}\label{prop:2-1}
Let $C$ be a plane curve of degree~$d$ defined by a homogeneous equation $F(\vect{x})=0$. Suppose that $a$ is a smooth (simple) point of~$C$.
\begin{enumerate}
\item For any point $b$ in $\P^{2}$ different from~$a$, the line $\ell=\overline{ab}$ is tangent to $C$ at $a$ if and only if $a\in C\cap P_{b}(C)$.
{\upshape (}See Figure~\ref{fig:1}.{\upshape)}
\item 
The equation of the tangent line $T_{a}(C)$ at $a$ is given by 
\[
\Dir_{\vect{x}}F(\vect{a})=0, \quad\text{or}\quad 
F'(\vect{a})\vect{x}=0.
\]
\item
$P_{a}(C)$ is tangent to $C$ at~$a$. 
\item
$a$ is an inflection point if and only if $a\in C\cap\He(C)$.\end{enumerate}
\end{prop}

\begin{proof}
(1) \ 
The line $\ell=\overline{ab}$ is tangent to $C$ at $a$ if and only if the  multiplicity of intersection at $a\in\ell \cap C$ is at least two. By \eqref{eq:Taylor-n}, this is equivalent to the condition $F(\vect{a})=0$ and $\Dir_{\vect{b}}F(\vect{a})=0$. This in turn is equivalent to the condition $a\in C\cap P_{b}(C)$. 

(2) \
A point $b$ is on the tangent line $T_{a}(C)$ if and only if $\ell=\overline{ab}$ is tangent to $C$ at $a$.  By the proof of (1), the latter is equivalent to $\Dir_{\vect{b}}F(\vect{a})=0$.  Thus, $\Dir_{\vect{x}}F(\vect{a})=0$ is the equation of~$T_{a}(C)$

(3) \
By Euler's formula, $F(\vect{a})=0$ implies $\Dir_{\vect{a}}F(\vect{a})=0$.  Thus, $P_{a}(C)$ passes through~$a$. 
Suppose $b$ is on $T_{a}(C)$.  Then, we have $\Dir_{\vect{b}}F(\vect{a})=0$ by~(2).  We would like to show that $\overline{ab}$ is also tangent to $P_{a}(C)$.  Applying Euler's formula to the polynomial $\Dir_{\vect{b}}F(\vect{x})$ of degree~$d-1$, we have 
\[
\Dir_{\vect{x}}(\Dir_{\vect{b}}F)(\vect{x})=(d-1)\Dir_{\vect{b}}F(\vect{x}). 
\]
Using the formula $\Dir_{\vect{x}}(\Dir_{\vect{b}}F)(\vect{x})=\Dir_{\vect{b}}(\Dir_{\vect{x}}F)(\vect{x})$ and replacing $\vect{x}$ by $\vect{a}$, we obtain 
\[
\Dir_{\vect{b}}(\Dir_{\vect{a}}F)(\vect{a})=(d-1)\Dir_{\vect{b}}F(\vect{a})=0. 
\]
This implies that $b$ is tangent at $a$ to the curve defined by $\Dir_{\vect{a}}F(\vect{x})=0$, which is nothing but $P_{a}(C)$.

(4) \ 
The line $\ell=\overline{ab}$ is an inflection tangent to $C$ at $a$ if and only if $F(\vect{a})=\Dir_{\vect{b}}F(\vect{a})=\Dir_{\vect{b}^{2}}F(\vect{a})=0$.  Thus, if $T_{a}(C)$ is an inflection tangent, any point $b\in T_{a}(C)$ satisfies the condition $\Dir_{\vect{b}^{2}}F(\vect{a})=0$.  This implies that the tangent line $\Dir_{\vect{x}}F(\vect{a})=0$ is contained in the curve defined by $\Dir_{\vect{x}^{2}}F(\vect{a})=0$ as a component. Since $\Dir_{\vect{x}^{2}}F(\vect{a})=\tr\vect{x}F''(\vect{a})\vect{x}$, $\Dir_{\vect{x}^{2}}F(\vect{a})=0$ is a conic, and this conic is degenerate if and only if $\det F''(\vect{a})=0$ by Lemma~\ref{lem:conic}(2).  Thus, $a\in C\cap \He(C)$. 

Conversely, suppose $a\in C\cap \He(C)$.  Since $\det F''(\vect{a})=0$, the conic $\Dir_{\vect{x}^{2}}F(\vect{a})=0$ is degenerate.  This conic passes through $a$ since $\Dir_{\vect{a}^{2}}F(\vect{a})=d(d-1)F(\vect{a})=0$ by Euler's formula (Lemma~\ref{lem:euler}).  Also, the tangent line $T_{a}(C):\Dir_{\vect{x}}F(\vect{a})=0$ is contained in the degenerate conic $\Dir_{\vect{x}^{2}}F(\vect{a})=0$.  This is because the tangent line to this conic at $a$ is given by $\tr\vect{x}F''(\vect{a})\vect{a}=0$ by Lemma~\ref{lem:conic}(3), and  $\tr\vect{x}F''(\vect{a})\vect{a}=\Dir_{\vect{a}}\Dir_{\vect{x}}F(\vect{a})=(d-1)\Dir_{\vect{x}}F(\vect{a})$ again by Euler's formula (applied to the polynomial $\Dir_{\vect{x}}F(\vect{\vect{y}})$ of degree $d-1$ in $\vect{y}$).  Thus, $T_{a}(C)$ is an inflection tangent and $a$ is an inflection point.
\end{proof}

From now on we focus on the case where $C$ is a cubic curve.  In this case the Taylor expansion formula around the point $(s:t)=(1:0)$ gives a homogenous cubic polynomial in $s$ and~$t$:
\begin{multline}\label{eq:Taylor-a}
F(s\vect{a}+t\vect{b})=
%\\
F(\vect{a}) s^{3}
+ \Dir_{\vect{b}}F(\vect{a}) s^{2}t
+\frac{1}{2!}\Dir_{\vect{b}^{2}}F(\vect{a}) st^{2}
+\frac{1}{3!}\Dir_{\vect{b}^{3}}F(\vect{a}) t^{3}.
\end{multline}
Exchanging the roles of $\vect{a}$ and $\vect{b}$ in \eqref{eq:Taylor-a}, that is, using the Taylor expansion  formula around the point $(s:t)=(0:1)$,
we have another form of expansion:
\begin{multline}\label{eq:Taylor-b}
F(s\vect{a}+t\vect{b})=
%\\
F(\vect{b}) t^{3} 
+ \Dir_{\vect{a}}F(\vect{b}) st^{2}
+\frac{1}{2!}\Dir_{\vect{a}^{2}}F(\vect{b}) s^{2}t
+\frac{1}{3!}\Dir_{\vect{a}^{3}}F(\vect{b}) s^{3}.
\end{multline}
Comparing the corresponding coefficients in \eqref{eq:Taylor-a} and \eqref{eq:Taylor-b}, we have
\begin{equation}\label{eq:symm}
\begin{gathered}
F(\vect{a}) = \frac{1}{3!}\Dir_{\vect{a}^{3}}F(\vect{b}), 
\qquad
\Dir_{\vect{b}}F(\vect{a}) 
= \frac{1}{2!}\Dir_{\vect{a}^{2}}F(\vect{b}), \\
\frac{1}{2!}\Dir_{\vect{b}^{2}}F(\vect{a}) 
= \Dir_{\vect{a}}F(\vect{b}),
\qquad
\frac{1}{3!}\Dir_{\vect{b}^{3}}F(\vect{a}) = F(\vect{b}).
\end{gathered}
\end{equation}
With matrix notation the second and the third relation may be written as follows:
\begin{gather}\label{eq:mat-notation}
F'(\vect{a}) \vect{b}
=\frac{1}{2!}\left(\tr\vect{a} F''(\vect{b})\vect{a}\right),
\quad
\frac{1}{2!}\left(\tr\vect{b} F''(\vect{a})\vect{b}\right)
= F'(\vect{b}) \vect{a}.
\end{gather}

\begin{prop}\label{prop:2-2}
Let $C$ be a plane cubic curve defined by an equation $F(\vect{x})=0$.
\begin{enumerate}
\item 
If $a$ is a smooth point of~$C$, then the equation of the tangent line $T_{a}(C)$ at $a$ may be written in two different forms
\[
\Dir_{\vect{x}}F(\vect{a})=0, \quad\text{and}\quad 
\Dir_{\vect{a^{2}}}F(\vect{x})=0.
\]
In particular, the second polar $P_{a^{2}}(C)$ of $C$ is the tangent line $T_{a}(C)$. 
\item
If $a$ is a singular point of~$C$, then $P_{a^{2}}(C)$ coincides with $\P^{2}$.
\end{enumerate}
\end{prop}

\begin{proof}
(1) \  
By Proposition~\ref{prop:2-1}(2), $\Dir_{\vect{x}}F(\vect{a})=0$ is the equation of $T_{a}(C)$. By \eqref{eq:symm} this equation is equivalent to $\Dir_{\vect{a^{2}}}F(\vect{x})=0$.  But, this is nothing but the equation of the second polar $P_{a^{2}}(C)$, and thus, $P_{a^{2}}(C)$ coincides with $T_{a}$.

(2) \
If $a$ is a singular point, the multiplicity of the intersection $\overline{ab}\cap C$ at $a$ is always greater than~$1$.  It follows from \eqref{eq:Taylor-a} that for any point $b\in\P^{2}$, $\Dir_{\vect{b}}F(\vect{a})=0$.  Then, by \eqref{eq:symm}, we have $\Dir_{\vect{a^{2}}}F(\vect{b})=0$ for any $b\in\P^{2}$, which implies $P_{a^{2}}(C)=\P^{2}$.
\end{proof}

\begin{figure}\label{fig:2}
\includegraphics[scale=0.36]{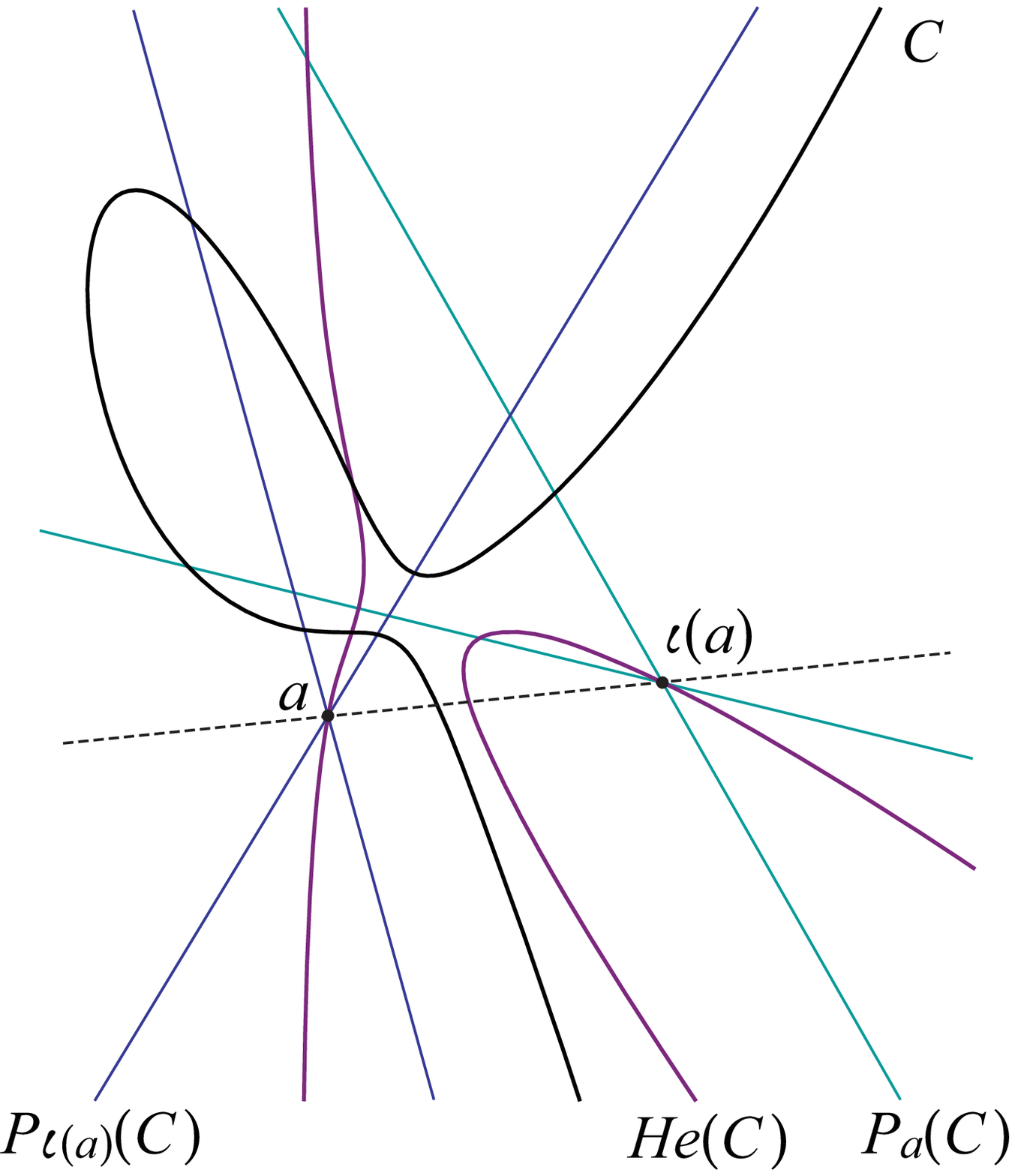}
\qquad\quad\includegraphics[scale=0.36]{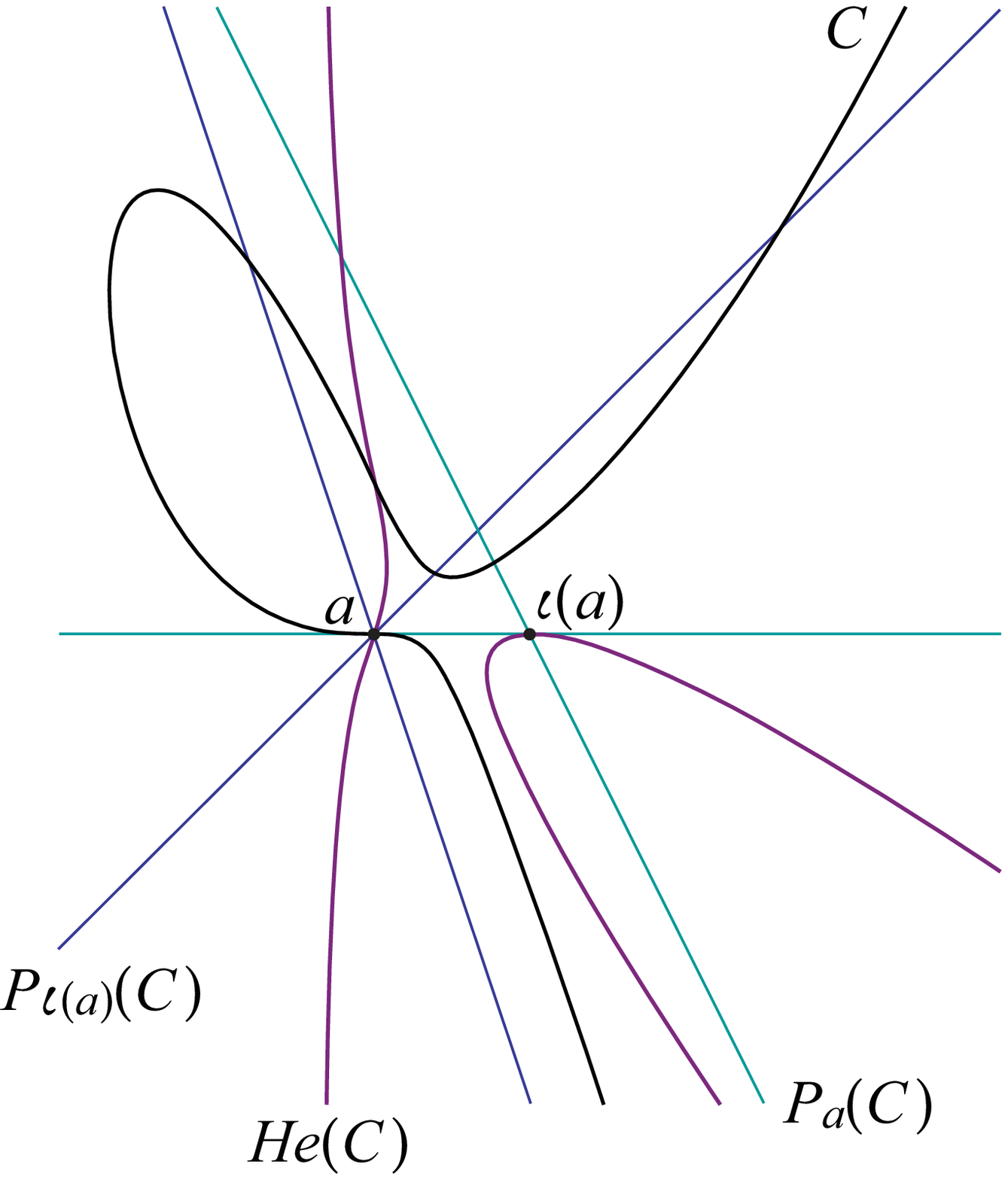}
\caption{Involution on the Hessian curve}
\end{figure}

\begin{prop}\label{prop:2-3}
Let $C$ be a plane cubic curve. Suppose that the Hessian $\He(C)$ does not coincide with $\P^{2}$.
\begin{enumerate} 
\item
A point $a\in \P^{2}$ is on $\He(C)$ if and only if the first polar $P_{a}(C)$ is degenerate.
\item 
Suppose $a$ is on $\He(C)$.  Let $b$ be a singular point of the degenerated first polar $P_{a}(C)$.  Then, $b$ is again on $\He(C)$, and $a$ is a singular point of $P_{b}(C)$.
%\item
%Suppose $a$ is a smooth point of~$C$.  Then, $a$ is an inflection point if and only if $a\in C\cap\He(C)$.
\end{enumerate}
\end{prop}

\begin{proof}
(1) \ 
Using \eqref{eq:symm}, we see that the equation of $P_{a}(C)$ can also be written in the form
\begin{equation}\label{eq:polar2}
\Dir_{\vect{x}^{2}}F(\vect{a})= \tr\vect{x}F''(\vect{a})\vect{x}=0.
\end{equation}
Thus, $P_{a}(C)$ is degenerate if and only if $\det F''(\vect{a}) =0$.

(2) \ 
If $a\in\He(C)$, it follows from (1) that $P_{a}(C)$ has a singular point. The first polar $P_{a}(C)$ is defined by the equation $\Dir_{\vect{a}}F(\vect{x})=0$.  By the Jacobian criterion, a singular point of $P_{a}(C)$ is a solution of the system of equations 
\[
\frac{\partial}{\partial x} \Dir_{\vect{a}}F(\vect{x})=
\frac{\partial}{\partial y} \Dir_{\vect{a}}F(\vect{x})=
\frac{\partial}{\partial z} \Dir_{\vect{a}}F(\vect{x})=0.
\]
With matrix notation these equations can be combined into one equation 
\begin{equation}\label{eq:criterion1}
F''(\vect{x})\vect{a}=\vect{0}.
\end{equation}
Now, if $b=(b_{0}:b_{1}:b_{2})$ is a singular point of $P_{a}(C)$, then $\vect{b}=\tr(b_{0},b_{1},b_{2})$ satisfies the equation~\eqref{eq:criterion1}, that is, we have $F''(\vect{b})\vect{a}=\vect{0}$.  This, in particular, implies that $\det F''(\vect{b})=0$.  This shows $b\in \He(C)$.

Meanwhile, $P_{b}(C)$ is given by the equation $\tr\vect{x}F''(\vect{b})\vect{x}=0$ just as $P_{a}(C)$ is given by~\eqref{eq:polar2}.  Thus, by Lemma~\ref{lem:conic}, a singular point of such a conic is a solution of the equation 
\begin{equation}\label{eq:criterion2}
F''(\vect{b})\vect{x}=\vect{0}.
\end{equation}
Then, the condition $F''(\vect{b})\vect{a}=\vect{0}$  can be interpreted that $\vect{a}$ satisfies the equation~\eqref{eq:criterion2}. This implies that $a$ is a singular point of $P_{b}(C)$.
\end{proof}

\begin{prop}\label{prop:involution}
Let $C$ be a plane cubic curve. Suppose that the Hessian $\He(C)$ is a   nonsingular cubic curve.
\begin{enumerate}
\item
If $a$ is on $\He(C)$, then $P_{a}(C)$ is the union of two distinct lines.
\item 
The map that associates $a\in\He(C)$ to the unique singular point $b$ of $P_{a}(C)$ determines an involution $\iota$ on $\He(C)$ without fixed points.
\item
If $a\in C\cap \He(C)$, then the inflection tangent line $T_{a}(C)$ is contained as a component in the degenerated first polar curve $P_{a}(C)$.  
\end{enumerate}
\end{prop}

\begin{proof}
(1) \ 
If $a$ is on $\He(C)$ and $P_{a}(C)$ is a double line or the entire plane, then all the points $b$ on $P_{a}(C)$ are singular points.  Then, by 
Proposition~\ref{prop:2-3}\,(2), $b$ is on $\He(C)$.  This means $P_{a}(C)$ is contained in $\He(C)$ as a component.  This contradicts the assumption that $\He(C)$ is nonsingular.

(2) \ 
Proposition~\ref{prop:2-3}\,(2) shows that the map $\iota$ described in the statement is indeed an involution.
It only remains to prove that this involution does not have any fixed point. If $a$ is a fixed point of $\iota$, then $a$ is the singular point of $P_{a}(C)$. 
Since $P_{a}(C)$ is a conic given by $\tr\vect{x}F''(\vect{a})\vect{x}=0$, this implies $F''(\vect{a})\vect{a}=\vect{0}$ by Lemma~\ref{lem:conic}(1).  In general, for a conic $Q$ given by $\tr\vect{x}M\vect{x}=0$ with a symmetric matrix $M$, its first polar $P_{a}(Q)$ is given by $\tr\vect{x}M\vect{a}=0$.  Thus, for any $\vect{x}$ we have
\[
\Dir_{\vect{a}^{2}}F(\vect{x})
=\Dir_{\vect{a}}\bigl(\tr\vect{x}F''(\vect{a})\vect{x}\bigr)
=\tr\vect{x}F''(\vect{a})\vect{a}
=0
\]
This means that $P_{a^{2}}(C)=\P^{2}$.  Then, for any $\vect{x}$ we have
\[
\Dir_{\vect{a}^{2}}F(\vect{x})=\Dir_{\vect{x}}F(\vect{a})=F'(\vect{a})\vect{x}=0.
\] 
This implies $F'(\vect{a})=0$, and thus $a$ is a singular point of $C$.  This contradits the assumption that $C$ is nonsingular.  Hence, the involution does not have a fixed point.  

(3) \ 
If $a\in C$, then $P_{a}(C)$ is tangent to $C$ at $a$ by Proposition~\ref{prop:2-1}(3), and thus $P_{a}(C)$ is tangent to $T_{a}(C)$ at~$a$.  If furthermore $a\in\He(C)$, then $P_{a}(C)$ is degenerate, and by~(2) the unique singular point $b=\iota(a)$ of $P_{a}(C)$ is different from~$a$.  Thus, if $a\in C\cap \He(C)$, $T_{a}(C)$ is a component of $P_{a}(C)$.
\end{proof}

\section{Cayleyan curve}

Let $C$ be a cubic curve, and $\He(C)$ its Hessian.  Throughout this section we consider the case where $\He(C)$ is a nonsingular curve.  Then, Proposition~\ref{prop:involution}\,(2) implies that $\He(C)$ admit a fixed-point-free involution $\iota$. 

\begin{definition}
Let $\eta$ be the map defined by 
\[
\eta:\He(C)\to \dP^{2}; \quad a\mapsto \overline{a\,\iota(a)}.
\]
The Cayleyan curve of $C$ is defined as the image of~$\eta$.
\end{definition}

It is easy to see that $\eta$ is an unramified double cover since $\iota$ is an involution without a fixed point.

Let $J_{H}=J(\He(C))$ be the Jacobian of $\He(C)$.  $J_{H}$ acts on $\He(C)$ by a translation. Choose one of the inflection points of $\He(C)$ as the origin $o$, and identify $\He(C)$ with its Jacobian~$J_{H}$.  A fixed point free involution corresponds to a translation by a point of order~$2$.  Let $\tau$ be the element of order~$2$ corresponding to the involution~$\iota$.  We have $\iota(a)=a+\tau$.   Then, we have a diagram:
\[
\xymatrix{
J_{H}=\He(C) \ar[r]^{\eta} \ar[d]
& \Cay(C) \subset\dP^{2}
\\ 
J_{H}/\<\tau\>=\He(C)/\<\iota\> \ar[ru]_{\cong} &} 
\]
Thus, $\Cay(C)$ may be identified with the quotient $\He(C)\to \He(C)/\<\iota\>$.

\begin{prop}\label{prop:3-1}
Let $C$ be a cubic curve defined by an equation $F(\vect{x})=0$ such that its Hessian $\He(C)$ is a nonsingular cubic curve. If $a$ is a point of $\He(C)$, 
then the second polar curve $P_{a^{2}}(C)$ is the tangent line to $\He(C)$ at $\iota(a)$.
\end{prop}

\begin{proof} 
$P_{a^{2}}(C)$ is a line given by $\Dir_{\vect{a}^{2}}F(\vect{x})=\tr\vect{a}F''(\vect{x})\vect{a}=0$ (see \eqref{eq:polar2}).  As $b\in P_{a^{2}}(C)$ moves on the line, we obtain a pencil of conics $\{P_{b}(C)\}_{b\in  P_{a^{2}}(C)}$.  It has four base points counting multiplicity.  Since for any $b\in P_{a^{2}}(C)$, $\vect{b}$  satisfies $\tr\vect{a}F''(\vect{b})\vect{a}=0$, and $P_{b}(C)$ is given by $\tr\vect{x}F''(\vect{b})\vect{x}=0$, we see that $a\in P_{b}(C)$ for any $b\in P_{a^{2}}(C)$.  This implies that $a$ is one of the base points. By the definition of $\iota(a)$, we have $F''(\iota(\vect{a}))\vect{a}=\vect{0}$.  In particular we have $\tr\vect{a}F''(\iota(\vect{a}))\vect{a}=0$ and $\iota(a)\in P_{b}(C)$.  Since $a$ is a double point of $P_{\iota(a)}(C)$ by Proposition~\ref{prop:2-3}(2), $a$ is a base point of the pencil with multiplicity~$2$.   For such a conic pencil having a base point with multiplicity~$2$, two of three degenerate conics in its members collapses into one multiply degenerated conic. In this case $P_{\iota(a)}(C)$ is such a multiply degenerate conic.  On the other hand, $P_{b}(C)$ is degenerate if and only if $b\in P_{a^{2}}(C) \cap \He(C)$.  Thus, $\iota(a)\in P_{a^{2}}(C)\cap \He(C)$ is an intersection point with multiplicity~$2$.  In other words, the line $P_{a^{2}}(C)$ is tangent to $\He(C)$ at $\iota(a)$. 
\end{proof}

\begin{prop}\label{prop:3-2}
Let $C$ be a nonsingular plane cubic curve whose Hessian $\He(C)$ is a nonsingular cubic curve, and let $a$ be a point in $C\cap \He(C)$.
\begin{enumerate}
\item
The inflection tangent line $T_{a}(C)$ is again tangent to $\He(C)$ at $\iota(a)$.  In particular, $T_{a}(C)$ coincides with the line $\overline{a\,\iota(a)}$.\item 
$a$ is also an inflection point of  $\He(C)$.
\end{enumerate}
\end{prop}

\begin{proof}
(1) \
By Proposition~\ref{prop:2-2}(1), the tangent line $T_{a}(C)$ equals $P_{a^{2}}(C)$. By Proposition~\ref{prop:3-1}, $P_{a^{2}}(C)$ is tangent to $\He(C)$ at $\iota(a)$.

(2) \ 
Choose one of the inflection points of $\He(C)$ as the origin $o$, and identify $\He(C)$ with its Jacobian~$J_{H}$.   By~(1) the line $\overline{a\,\iota(a)}$ is tangent to $\He(C)$ at $\iota(a)$.  This translates to the equation $a+2\iota(a)=o$. On the other hand, we have $a+2\iota(a)=a+2(a+\tau)=3a+2\tau=3a$.  Thus, we have $3a=o$, which implies that $a$ is an inflection point of $\He(C)$.
\end{proof}

\begin{prop}\label{prop:3-3}
Let $C$ be a nonsingular plane cubic curve whose Hessian $\He(C)$ is a nonsingular cubic curve.  A line $l\in \dP^{2}$ belongs to $\Cay(C)$ if and only if it is an irreducible component of the first polar curve $P_{d}(C)$ for some $d\in \He(C)$.
\end{prop}

\begin{proof}
Let $l$ be the line $\overline{a\,\iota(a)}\in \Cay(C)$, where $a\in\He(C)$.  Then, by Propositon~\ref{prop:3-1} the second polar curves $P_{a^{2}}(C)$ and $P_{\iota(a)^{2}}(C)$ are tangent to $\He(C)$ at $\iota(a)$ and $a$ respectively.
Identifying $\He(C)$ with its Jacobian $J_{H}$ as before, $P_{a^{2}}(C)$, $P_{\iota(a)^{2}}(C)$ and $\He(C)$ converge at the point corresponding to $-2a$.  
Put $d=-2a$. Then, $P_{d}(C)$ is the union of two distinct lines intersecting at $d+\tau=-2a+\tau$, which is the third point of intersection between $l$ and $\He(C)$.  From the proof of Proposition~\ref{prop:3-1} we see that $P_{d}(C)$ is a member of a conic pencil passing through $a$.  Thus, one of the irreducible components of $P_{d}(C)$ passes through $a$ and $-2a+\tau$, and thus coincides with~$l$.

Conversely, consider $P_{d}(C)$ for any point~$d$.  It is the union of two distinct lines intersecting at $d+\tau$.  There are four points satisfying the equation $-2x=d$. These four solutions are written in the form $a$, $a+\tau$, $a+\tau'$, and $a+\tau'+\tau$, where $\tau'$ is another point of order~$2$ of $\He(C)$.  Now the argument of the first half of the proof shows that the components of $P_{d}(C)$ are $\overline{a\,\iota(a)}$ and $\overline{a'\,\iota(a')}$, where $a'=a+\tau'$.  This completes the proof.
\end{proof}

Let $l$ be a point of $\Cay(C)$.  Then, from the proof of Proposition~\ref{prop:3-3}, $l$ is a component of $P_{d}(C)$ for $d=-2a\in \He(C)$.  Let $\iota'$ be the map that associates to $l$ the other component of $P_{d}(C)$.

\begin{prop}\label{prop:3-4}
The map $\iota'$ is an involution of $\Cay(C)$ without fixed points. It corresponds to the translation by the nontrivial element $[\tau']$ of $J_{H}[2]/\<\tau\>$.
\end{prop}

\begin{proof}
It is clear that $\iota'$ is an involution.  It has no fixed point by Proposition~\ref{prop:involution}(1).  From the last part of the proof of Proposition~\ref{prop:3-3}, we see that $\iota'(\overline{a\,\iota(a)})$ is obtained by adding $\tau'$ to $a$.  The second part follows from this.
\end{proof}

\begin{prop}\label{prop:3-5}
Let $a\in C\cap \He(C)$ be an inflection point of $C$ and $\He(C)$.  Let $T_{l}(\Cay(C))$ be the tangent line at $l=\overline{a\,\iota(a)}\in \Cay(C)$, and let $l'\in \Cay(C)$ be the third point of intersection between $T_{l}(\Cay(C))$ and $\Cay(C)$.  Then, $l'$ is an inflection point of $\Cay(C)$.
\end{prop}

\begin{proof}
Identify $\He(C)$ with its Jacobian $J_{H}$ by choosing $a$ as the origin of the group structure.  Then $\Cay(C)$ is identified with $J_{H}/\<\tau\>$, and $l$ is the origin of $\Cay(C)$.  

We claim that the tangent line $T_{l}(\Cay(C))$ corresponds to a pencil of lines in $\P^{2}$ centered at $\iota(a)$.  In general, a line in $\dP^{2}$ corresponds to a pencil of lines in $\P^{2}$ centered at a point. 
If $b\in \He(C)$, then three lines among the pencil of lines centered at $b$ belong to $\Cay(C)$; these are $\overline{b\,\iota(b)}$, $\overline{b'\,\iota(b')}$ and $\overline{b''\,\iota(b'')}$, where $b'$ and $b''$ are points satisfying the equation $-2x+\tau=b$.  For $b=\iota(a)$, the line $\overline{b\,\iota(b)}=\overline{a\,\iota(a)}$ and one of the other two lines coincide since $-2a+\tau=\iota(a)$.  This shows that the line in $\dP^{2}$ corresponding to the pencil of lines in $\P^{2}$ centered at $\iota(a)$ is tangent to $\Cay(C)$.

The first polar curve $P_{a}(C)$ is the union of two lines passing through $\iota(a)$, and both lines are contained in $\Cay(C)$ by Proposition~\ref{prop:3-3}. The third point of intersection $l'\in T_{l}(\Cay(C))\cap \Cay(C)$ corresponds to the line other than $\overline{a\,\iota(a)}$.  This implies that $l'$ is a point of order~$2$, namely $2l'=l$.

For $l_{1}, l_{2}\in \Cay(C)$, let $l_{1}*l_{2}$ be the third point of intersection between the line $\overline{l_{1}\,l_{2}}$ and $\Cay(C)$.  With this notation, we have $l*l=l'$.  The condition $2l'=l$ is equivalent to $l*(l'*l')=l$.  Since $l_{1}*l_{2}=l_{3}$ implies $l_{2}=l_{1}*l_{3}$ in general, $l*(l'*l')=l$ implies $l'*l'=l*l$.  Thus, we have $l'*l'=l'$, which shows that $l'$ is an inflection point.
\end{proof}

\begin{cor}\label{cor:prop3-5}.
Let $a\in C\cap \He(C)$ be an inflection point of $C$ and $\He(C)$, and let $l$ be the line $\overline{a\,\iota(a)}\in \Cay(C)$.  Then, $\iota'(l)$ is an inflection point of $\Cay(C)$.
\end{cor}

\begin{proof}
From the above proof, we see that the line $\overline{a\,\iota(a)}$ is a component of $P_{a}(C)$ and the point $l'\in T_{l}(\Cay(C))\cap \Cay(C)$ corresponds to the other component of $P_{a}(C)$.  This implies that $l'=\iota'(l)$, and thus $\iota'(l)$ is an inflection point of $\Cay(C)$.
\end{proof}

\section{Symplectically isomorphic family via Hessian}

In this section the base field $k$ is assumed to be a number field.
Let $E_{0}$ be an elliptic curve defined over~$k$.  To apply the classical theory developed in the previous sections, we choose a model of $E_{0}$ as a plane cubic curve such that the origin $O$ is an inflection point.  
With this choice, we have the property that three points $P$, $Q$ and $R$ are collinear if and only if $P+Q+R=O$.  In particular, the inflection points corresponds to  $3$-torsion points $E_{0}[3]$.  

Any line joining two inflection points $T$ and $T'$ intersects with $E_{0}$ at another inflection point $T''$.  The set $\{T,T',T''\}$ is a coset with respect to a subgroup of $E_{0}[3]$.  Since there are four subgroups of order three in $E_{0}[3]\cong (\Z/3\Z)^{2}$, there are twelve lines each of which contains three inflection points.  

Consider the pencil of cubic curves $E_{0}+ t \He(E_{0})$, or more precisely the pencil of cubic curves defined by the equation
\[
F(x,y,z) + t \det\bigl(F''(x,y,z)/2!\bigr) = 0,
\]
where $F(\vect{x})=0$ is the equation of $E_{0}$.  The nine base points of this linear system are the inflection points of $E_{0}$ (and also of $\He(E_{0})$ by Propositon~\ref{prop:3-2}.)  Blowing up at these nine base points simultaneously, we obtain an elliptic surface $\E_{t}\to \P^{1}$ defined over~$k$.  By an abuse of notation we use $\E_{t}$ to indicate the pencil of cubic curves and also this elliptic surface.

\begin{prop}\label{prop:4-1}
The elliptic surface $\E_{t}$ is a rational elliptic surface which has four singular fibers of type I${}_{3}$.  It is of type No.~68 in the Oguiso-Shioda classification table (\cite{Oguiso-Shioda}).  It is isomorphic over $\bar k$ to the Hesse pencil 
\[
x^{3}+y^{3}+z^{3}=3\lambda xyz.
\]
\end{prop}

\begin{proof}
It is obvious that $\E_{t}$ is a rational surface, as it is obtained by blowing up $\P^{2}$.  
Let $G$ be a subgroup of order~$3$ of $E_{0}[3]$.  Then, for each coset of $G$ there is a line passing through three points contained in the coset.  These three lines form a singular fiber of type either I${}_{3}$ or IV.  There are four such fibers.  Counting the Euler numbers, all of these four fibers must be of type I${}_{3}$ and there are no other singular fibers. Such surface is classified as No.~68 in Oguiso-Shioda classification.  Beauville \cite{Beauville} shows that such an elliptic surface must be isomorphic to the Hesse pencil over $\bar k$.
\end{proof}

\begin{theorem}\label{thm:4-1}
Let $E_{0}$ be an elliptic curve defined over $k$ given by a homogeneous cubic equation $F(x,y,z)=0$ in $\P^{2}$ such that the origin $O$ is one of the inflection point.  Let $\E_{t}$ be the pencil of cubic curves defined by
\[
\E_{t}: F(x,y,z)+t\,\det\bigl(F''(x,y,z)/2!\bigr)=0.
\]
Then, the identity map $(x:y:z)\mapsto (x:y:z)$ gives a symplectic isomorphism $E_{0}[3]\to\E_{t}[3]$ for each $t$ such that $\E_{t}$ is an elliptic curve.  Any elliptic curve $E$ over~$k$ with a symplectic isomorphism $\phi:E_{0}[3]\to E[3]$ is a member of $\E_{t}$.
\end{theorem}

\begin{proof}
Since the base points of the pencil are the sections of the associated elliptic surface, and the Mordell-Weil group of our elliptic surface is isomorphic to $(\Z/3\Z)^{2}$, the identity map $(x:y:z)\mapsto (x:y:z)$ restricted to the inflection points ($=$ base points) gives a symplectic isomorphism $E_{0}[3]\to\E_{t}[3]$.

Since $\E_{t}$ is a twist of the Hesse pencil, it is a universal curve if we viewed it as a curve over $\P^{1}$ minus four points at which the fibers are singular.  The last assertion follows from this immediately.
\end{proof}

Let us write down the explicit equations.  We assume that $E_{0}$ is given by the Weierstrass equation 
\[
E_{0}:y^{2}z = x^{3} + Axz^{2} + B z^{3}.
\]
If we choose another model, computations can be done in a similar way.

The Hessian of the curve $E_{0}$ is given by
\[
\He(E_{0}):
\left|\begin{array}{ccc}
-3x & 0 & -Az\\
 0 & z & y \\ 
-Az & y & -Ax-3Bz
\end{array}\right|
=3Ax^2z+9Bxz^2+3xy^2-A^2z^3=0.
\]
A simple calculation shows that $\He(E)$ is singular if and only if $A(4A^{3}+27B^{2})=0$.

Note that the Hessian of $\E_{t}$ is of the form $\E_{t_{H}}$, where 
\[
t_{H}=\frac{-27Bt^3+9At^2+1}{9t(3A^2t^2+9Bt-A)}.
\]
This implies that the nine base points are inflection points of each smooth member of $\E_{t}$.

\begin{theorem}
Let $E_{0}$ be an elliptic curve given by $y^{2}z=x^{3}+Axz^{2}+Bz^{3}$ with $A\neq0$.  Then, the nine base points of the pencil of cubic curves
\[
\E_{t}:(y^{2}z - x^{3} - Axz^{2} - B z^{3}) + t\,(3Ax^2z+9Bxz^2+3xy^2-A^2z^3)=0
\]
are inflection points of each member of the pencil that is nonsingular.  Thus, the identity map $(x:y:z)\mapsto (x:y:z)$ gives a symplectic isomorphism $E_{0}[3]\to\E_{t}[3]$ for each $t$ such that $\E_{t}$ is an elliptic curve.
Moreover, $\E_{t}$ is a universal family of elliptic curves $E$ over~$k$ with a symplectic isomorphism $\phi:E_{0}[3]\to E[3]$.  \qed
\end{theorem}

\begin{remark}
The Weierstrass form of $\E_{t}$ is given by
\begin{equation}\label{eq:Weier-E}
Y^{2}=X^{3}+a(t)X+b(t),
\end{equation}
where 
\begin{multline*}
\begin{aligned}
a(t) &= -27(A^3+9B^2)\,t^4+54AB\,t^3-18A^2t^2-18Bt+A,
\\
b(t) &= -243B(A^3+6B^2)\,t^6+54A(2A^3+9B^2)\,t^5
\end{aligned}
\\
+135A^2B\,t^4+270B^2\,t^3-45AB\,t^2+4A^2\,t+B.
\end{multline*}
The family of Rubin-Silverberg \cite{Rubin-Silverberg} and our family are related as follows.  Let $t_{RS}$ be the parameter of Rubin-Silverberg family, then our $t$ is given by 
\[
t = \frac{6 A B\,t_{RS}}{27 B^2\,t_{RS} +(4A^3+27B^2)}.
\]
\end{remark}

\section{Anti-symplectically isomorphic family via Cayleyan}

As in \S4, let $E_{0}$ be an elliptic curve defined over~$k$ realized as a plane cubic curve such that the origin $O$ is an inflection point. Using the same origin $O$, identify $\He(E_{0})$ with its Jacobian.  
Let $\eta:\He(E_{0})\to \Cay(E_{0})$ be the map $P\mapsto \overline{P\,\iota(P)}$, and $\iota'$ the involution on $\Cay(E_{0})$.  Then, by Corollary~\ref{cor:prop3-5}, the point $\iota'(\eta(O))$ is an inflection point of $\Cay(E_{0})$.  We denote $\iota'(\eta(O))$ by $O'$ and choose it as the origin of $\Cay(E_{0})$.  By this identification, $\Cay(E_{0})[3]$ is the set of inflection points of $\Cay(E_{0})\subset \dP^{2}$.  

\begin{prop}\label{prop:5-1}
The map $\phi$ that associates to $P\in E_{0}[3]$ the point $\iota'(\eta(P))\in \Cay(C)[3]$ gives an anti-symplectic isomorphism $\phi:E_{0}[3]\to \Cay(C)[3]$.
\end{prop}

The proof based on the following lemma.

\begin{lemma}[{Silverman \cite[Prop.~8.3]{Silverman}}]\label{lem:Weil-paring}
Let $\phi:E_{1}\to E_{2}$ be an isogeny, and let $P\in E_{1}[m]$ and $Q\in E_{2}[m]$.  Then the Weil pairings satisfy
\[
e_{E_{1},m}(P,\hat\phi(Q))=e_{E_{2},m}(\phi(P),Q),
\]
where $\hat\phi:E_{2}\to E_{1}$ is the dual isogeny.\qed
\end{lemma}

\begin{proof}[Proof of Proposition~\ref{prop:5-1}]
Let $\phi_{0}$ be the isogeny $\He(E_{0})[3]\to (\He(E_{0})/\<\tau\>)[3]$ of degree~$2$. Then, we have $\hat\phi_{0}\circ\phi_{0}=[2]$.  Thus, it follows form the above lemma that for $P_{1},P_{2}\in E_{1}[3]$
\begin{align*}
e_{E_{2},3}(\phi_{0}(P_{1}),\phi_{0}(P_{2}))
&=e_{E_{1},3}(P_{1},\hat\phi_{0}(\phi_{0}(P_{2})))
=e_{E_{1},3}(P_{1},[2]P_{2})
\\
&=e_{E_{1},3}(P_{1},P_{2})^{2}
=e_{E_{1},3}(P_{1},P_{2})^{-1}.
\end{align*}
This shows that $\phi_{0}$ is an anti-symplectic isomorphism.  The map $\phi$ is the composition of the identity map $E_{0}[3]\to \He(E_{0})$, the quotient map $\He(E_{0})[3]\to (\He(E_{0})/\<\tau\>)[3]$, and the translation $\iota'$.  Thus, $\phi$ is also an anti-symplectic isomorphism.
\end{proof}

As in previous section, we assume that $E_{0}$ is given by the Weierstrass equation \(y^{2}z = x^{3} + Axz^{2} + B z^{3}\).  Recall that the Hessian is given by \(3Ax^2z+9Bxz^2+3xy^2-A^2z^3=0\).  Let $P=(x_{0}:y_{0}:z_{0})$ be a point of $\He(E_{0})$, and let $\iota(P)=(x_{1}:y_{1}:z_{1})\in\He(E_{0})$.  Then, we have
\[
\left(\begin{array}{ccc}
-3x_{0} & 0 & -Az_{0}\\
 0 & z_{0} & y_{0} \\ 
-Az_{0} & y_{0} & -Ax_{0}-3Bz_{0}
\end{array}\right)
\left(\begin{array}{c} x_{1}\\y_{1}\\z_{1}\end{array}\right)
=\left(\begin{array}{c} 0\\ 0 \\ 0\end{array}\right)
\]
Thus, we obtain $(x_{1}:y_{1}:z_{1})=(Az_{0}^2:3x_{0}y_{0}:-3x_{0}z_{0})$ except for $x_{0}=z_{0}=0$, and for $(x_{0}:y_{0}:z_{0})=(0:1:0)$, we obtain $(x_{1}:y_{1}:z_{1})=(1:0:0)$.
The equation of the line $\overline{P\,\iota(P})$ is given by 
\[
\left|\begin{array}{ccc}
x_{0} & \phantom{-}Az_{0}^2 & x\\
y_{0} & \phantom{-}3x_{0}y_{0} & y \\ 
z_{0} & -3x_{0}z_{0} & z
\end{array}\right|=0.
\]
Thus, we have 
\begin{multline*}
\phi:(x_{0}:y_{0}:z_{0})\mapsto (\xi_{0}:\eta_{0}:\zeta_{0})
%\\
=\bigl(6x_{0}y_{0}z_{0}:
-z_{0}(3x_{0}^{2}+Az_{0}^{2}):-y_{0}(3x_{0}^{2}-Az_{0}^{2})\bigr).
\end{multline*}
Eliminating $x_{0},y_{0},z_{0}$ using the equation of $\He(E_{0})$, we see that $\xi_{0},\eta_{0}$, and $\zeta_{0}$ satisfy the relation
$A\xi_{0}^3+3\xi_{0}\zeta_{0}^2+3(3B\xi_{0}-2A\zeta_{0})\eta_{0}^2 =0$.
Thus, the equation of $\Cay(E_{0})\subset \dP^{2}$ is given by
\[
\Cay(E_{0}):A\xi^3+3\xi\zeta^2+3(3B\xi-2A\zeta)\eta^2 =0.
\]
$O=(0:1:0)\in \He(E_{0})$ is mapped to $(0:0:1)\in\Cay(E_{0})$ by $\phi$. 
The tangent line of $\Cay(E_{0})$ at $(0:0:1)$ is $\xi=0$, and the third point of intersection is $(0:1:0)$.  By Proposition~\ref{prop:3-5}, $(0:1:0)$ is an inflection point.  The tangent line at $(0:1:0)$ is given by $3B\xi - 2A\zeta=0$.
The change of variables
\begin{equation}\label{eq:change-of-var}
\xi'=3B\xi-2A\zeta,\quad \eta'=2A\eta,\quad \zeta'=-\xi,
\end{equation}
change the equation of $\Cay(E_{0})$ to
\begin{equation}\label{eq:Cayleyan}
\Cay(E_{0}):
-3\xi'^2\zeta'-18B\xi'\zeta'^2+3\xi'\eta'^2-(4A^3+27B^2)\zeta'^3=0,
\end{equation}
and the inflection tangent line at $(0:1:0)$ becomes $\xi'=0$.
Comparing with the equation of $\He(E_{0})$, we notice that the above equation 
is the Hessian of the curve
\[
\delta_{E_{0}}\eta'^2\zeta'
=\xi'^3-\delta_{E_{0}}\xi'\zeta'^2-2B\delta_{E_{0}}\zeta'^3,
\]
where $\delta_{E_{0}}=4A^3+27B^2$.

\begin{theorem}\label{thm:5-1}
Let $E_{0}$ be an elliptic curve given by $y^{2}z = x^{3} + Axz^{2} + B z^{3}$ with $A\neq0$, 
and let $F_{0}$ be the elliptic curve defined by the equation
\[
F_{0}:\delta_{E_{0}}y^2z
=x^3-\delta_{E_{0}}xz^2-2B\delta_{E_{0}}z^3,
\]
where $\delta_{E_{0}}=4A^3+27B^2$.  Let $\F_{t}$ be the pencil of cubic curves given by
\begin{multline*}
\F_{t}: (\delta_{E_{0}}y^2z
- x^3 + \delta_{E_{0}}xz^2 + 2B\delta_{E_{0}}z^3)
%\\
+t\,(-3x^2z-18Bxz^2+3xy^2-\delta_{E_{0}}z^3)=0.
\end{multline*}
Then, the map 
\[
\phi:(x:y:z)\mapsto \bigl(-y(3Ax^2+9Bxz-A^2z^2):Az(3x^2+Az^2):3xyz\bigr)
\]
gives a anti-symplectic isomorphism $E_{0}[3]\to\F_{t}[3]$ for each $t$ such that $\F_{t}$ is an elliptic curve.  Any elliptic curve $F$ over~$k$ with a symplectic isomorphism $\phi:E_{0}[3]\to F[3]$ is a member of $\F_{t}$.
\end{theorem}

\begin{proof}
The map $\phi$ above is the map appeared in Proposition~\ref{prop:5-1}, namely the composition of the identity $E_{0}[3]\to\He(E_{0})[3]$, the quotient map $\phi:\He(E_{0})\to\He(E_{0})/\<\tau\>$ and the translation of $\Cay(E_{0})$ by a point of order~$2$.  This map sends inflection points of $\He(E_{0})$ to those of $\Cay(E_{0})$, and this map is anti-symplectically isomorphic by Propsition~\ref{prop:5-1}.

The curve $\F_{t}$ is an universal family of elliptic curves whose $3$-torsion subgroup is simplectically isomorphic to $\Cay(E_{0})[3]$. This means that $\F_{t}$ is an universal family of elliptic curves whose $3$-torsion subgroup is simplectically isomorphic to $E_{0}[3]$
\end{proof}

\begin{remark}
If we replace $t$ by  $2A\,t/(9B\,t-\delta_{E})$, then 
the Weierstrass form of $\F_{t}$ becomes particularly simple  Namely, the Weierstrass forms of the elliptic pencil $\F'_{t}:(9B\,t-2A)F_{0}-\delta_{E}\,t\Cay(E_{0})$ is given by
\begin{equation}\label{eq:Weier-F}
\F'_{t}:-\delta_{E}\,Y^{2}=X^{3}+a(t)X+b(t),\qquad 
\end{equation}
where 
\begin{multline*}
\begin{aligned}
a(t) &= \delta_{E}\bigl(27A^2\,t^4+108B\,t^3-18A\,t^2-1\bigr),
\\
b(t) &= 2\delta_{E}\bigl(-243B(A^3+6B^2)\,t^6+54A(2A^3+9B^2)\,t^5 \end{aligned}
\\
+135A^2B\,t^4 +270B^2\,t^3-45AB\,t^2+4A^2\,t+B\bigr).
\end{multline*}
If we denote by $j_{\E_{t}}(t)$ the $j$-invariant of the elliptic surface given by~\eqref{eq:Weier-E}, and by $j_{\F'_{t}}(t)$ that of~\eqref{eq:Weier-F}.  Then, we have 
\[
j_{\F'_{t}}(t)/1728=1728/\,j_{\E_{t}}(t).
\]
In particular, we have $j_{F_{0}}/1728=1728/j_{E_{0}}$.
\end{remark}
\begin{remark}
If $E_{0}$ is given by the Hesse pencil $x^{3}+y^{3}+z^{3}=3\lambda\,xyz$, then $\He(E_{0})$, $\Cay(E_{0})$ and $F_{0}$ are given by thte following (cf. Artebani and Dolgachev \cite{Artebani-Dolgachev}).
\[\renewcommand{\arraystretch}{1.6}
\begin{array}{ccl}
\He(E_{0})&:&x^{3}+y^{3}+z^{3}=\frac{4-\lambda^3}{\lambda^{2}}\,xyz,
\\
\Cay(E_{0})&:&\xi^{3}+\eta^{3}+\zeta^{3}
=\frac{\lambda^3+2}{\lambda}\,\xi\eta\zeta,
\\
F_{0} &:&\xi^{3}+\eta^{3}+\zeta^{3}=-\frac{6}{\lambda}\,\xi\eta\zeta.
\end{array}\]

\end{remark}

\section{Applications}

If we view our families $\E_{t}$ and $\mathcal{F}_{t}$ as elliptic surfaces, it is apparent that they are rational elliptic surfaces over~$k$.  As a consequence, we are able to apply Salgado's theorem \cite{Salgado} to our family. 

\begin{theorem}\label{thm:salgado}
Let $E_{0}$ be an elliptic curve over~$k$. 
\begin{enumerate}
\item 
There are infinitely many elliptic curves $E$ over $k$ such that $E[3]$ is symplectically isomorphic to $E_{0}[3]$ and $\rank E(k)\ge 2$. 
\item 
There are inifinitely many elliptic curves $F$ over~$k$ such that $F[3]$ is anti-symplectically isomorphic to $E_{0}[3]$ and $\rank F(k)\ge 2$. 
\qed
\end{enumerate}
\end{theorem}

Since $E_{0}[3]$ and $\Cay(E_{0})[3]$ are anti-symplectically isomorphic to each other, Frey and Kani \cite{Frey-Kani} predict that there exists a curve $C$ of genus~$2$ that admits two morphism $C\to E_{0}$ and $C\to \Cay(E_{0})$ of degree~$3$.  Indeed, we have the following.

\begin{prop}\label{prop:genus-2}
Let $E_{0}$ be an elliptic curve over~$k$ given by $y^{2}=x^{3}+Ax+B$, and assume $A\neq0$.  The Weierstrass form of $\Cay(E_{0})$ is given by
\[
\Cay(E_{0}):-3y^{2}=x^3-18Bx^2+3\delta_{E}x,
\]
where $\delta_{E_{0}}=4A^3+27B^2$. Then, the curve $C$ given by
\[
C:Y^2=-(3X^2+4A)(X^3+AX+B)
\]
is a curve of genus~$2$ admitting two morphisms $\psi_{1}:C\to E_{0}$ and $\psi_{2}:C\to \Cay(E_{0})$ of degree~$3$. 
The maps $\psi_{1}:C\to E_{0}$ and $\psi_{2}:C\to \Cay(E_{0})$ are given by
\begin{gather*}
\psi_{1}:(X,Y)\mapsto 
(x,y)=\left(-\frac{X^3+4B}{3X^2+4A},
 \frac{(X^3+4AX-8B)Y}{(3X^2+4A)^2}\right),
\\
\psi_{2}:(X,Y)\mapsto 
(x,y)=\left(-\frac{\delta_{E_{0}}}{3(X^3+AX+B)},
\frac{\delta_{E_{0}}(3X^2+A)Y}{9(X^3+AX+B)^2}\right). 
\end{gather*}
\end{prop}

\begin{remark}
This is the degenerated case in the sense that $\psi_{2}$ is ramified at one place $X=\infty$ with ramification index 3. See Shaska \cite{Shaska} for more detail.
\end{remark}

%\bibliography{Hessian}{}

\begin{thebibliography}{1}

\bibitem{Artebani-Dolgachev}
Artebani, M.~and Dolgachev,  I.~V. \emph{The {H}esse pencil of plane cubic
  curves}, Enseign. Math. (2) \textbf{55} (2009), no.~3-4, 235--273.
  \MR{2583779 (2011d:14096)}

\bibitem{Beauville}
Beauville, A. \emph{Les familles stables de courbes elliptiques sur {${\bf
  P}^{1}$}\ admettant quatre fibres singuli\`eres}, C. R. Acad. Sci. Paris
  S\'er. I Math. \textbf{294} (1982), no.~19, 657--660. \MR{664643 (83h:14008)}

\bibitem{Dolgachev}
Dolgachev, I.~V. \emph{Classical algebraic geometry: A modern view}, Cambride
  Univ. Press, to be published.

\bibitem{Frey-Kani}
Frey, G.~and Kani, E. \emph{Curves of genus {$2$} covering elliptic
  curves and an arithmetical application}, Arithmetic algebraic geometry
  ({T}exel, 1989), Progr. Math., vol.~89, Birkh\"auser Boston, Boston, MA,
  1991, pp.~153--176. \MR{1085258 (91k:14014)}

\bibitem{Oguiso-Shioda}
Oguiso, K.~and Shioda, T. \emph{The {M}ordell-{W}eil lattice of a
  rational elliptic surface}, Comment. Math. Univ. St. Paul. \textbf{40}
  (1991), no.~1, 83--99. \MR{1104782 (92g:14036)}

\bibitem{Rubin-Silverberg}
Rubin, K.~and Silverberg, A. \emph{Families of elliptic curves with constant mod
  {$p$} representations}, Elliptic curves, modular forms, \& {F}ermat's last
  theorem ({H}ong {K}ong, 1993), Ser. Number Theory, I, Int. Press, Cambridge,
  MA, 1995, pp.~148--161. \MR{1363500 (96j:11078)}

\bibitem{Salgado}
Salgado, C. \emph{Rank of elliptic surfaces and base change}, C. R. Math.
  Acad. Sci. Paris \textbf{347} (2009), no.~3-4, 129--132. \MR{2538098
  (2010g:11098)}

\bibitem{Shaska}
Shaska, T.~\emph{Genus 2 fields with degree 3 elliptic subfields}, Forum Math.
  \textbf{16} (2004), no.~2, 263--280. \MR{2039100 (2004m:11097)}

\bibitem{Silverman}
Silverman, J.-H. \emph{The arithmetic of elliptic curves}, Graduate Texts
  in Mathematics, vol. 106, Springer-Verlag, New York, 1986. \MR{817210
  (87g:11070)}

\end{thebibliography}
%\bibliographystyle{amsplain}
\providecommand{\bysame}{\leavevmode\hbox to3em{\hrulefill}\thinspace}
\renewcommand{\MR}[1]{\relax}

\end{document}